\documentclass[a4paper]{amsart}
\usepackage[utf8]{inputenc}
\usepackage[ngerman, english]{babel}
\usepackage{amsmath,amsthm,amssymb,amsfonts}
\usepackage{mathrsfs}
\usepackage{ifthen}
\usepackage{enumerate}
\usepackage{comment}
\usepackage[pdfauthor={Martin Halla},pdftitle={Electromagnetic Stekloff eigenvalues: approximation analysis},
unicode,colorlinks=true,linkcolor=black,citecolor=black,urlcolor=black,pagebackref=false,breaklinks]{hyperref}
\usepackage{hyphsubst}
\usepackage{xcolor}
\usepackage{graphicx}
\usepackage{psfrag}
\usepackage{pst-all}

\newtheorem{theorem}{Theorem}[section]
\newtheorem{lemma}[theorem]{Lemma}
\newtheorem{corollary}[theorem]{Corollary}

\newtheorem{assumption}[theorem]{Assumption}

\newcommand{\Xt}{\tilde X}
\newcommand{\At}{\tilde A}
\newcommand{\Tt}{\tilde T}
\newcommand{\Zt}{\tilde Z}
\newcommand{\Vt}{\tilde V}
\newcommand{\Wt}{\tilde W}
\newcommand{\Pt}{\tilde P}

\newcommand{\dtilde}[1]{\tilde{\raisebox{0pt}[0.85\height]{$\tilde{#1}$}}}
\newcommand{\Xtt}{\dtilde X}
\newcommand{\Att}{\dtilde A}
\newcommand{\Ttt}{\dtilde T}
\newcommand{\Ptt}{\dtilde P}

\DeclareMathOperator{\I}{I}

\newcommand{\dd}{\mathrm{d}}

\newcommand{\nv}{\nu}

\newcommand{\ol}[1]{\overline{#1}}

\newcommand{\spl}{\langle}
\newcommand{\spr}{\rangle}
\newcommand{\bpm}{\begin{pmatrix}}
\newcommand{\epm}{\end{pmatrix}}

\renewcommand{\div}{\operatorname{div}}
\DeclareMathOperator{\curl}{curl}

\DeclareMathOperator{\tr}{tr}

\DeclareMathOperator{\spn}{span}

\renewcommand{\dim}{\operatorname{dim}}
\DeclareMathOperator{\ran}{ran}

\newcommand{\setC}{\mathbb{C}}

\newcommand{\setK}{\mathbb{K}}

\newcommand{\setN}{\mathbb{N}}

\newcommand{\setR}{\mathbb{R}}

\newcommand{\boldH}{\mathbf{H}}

\newcommand{\boldL}{\mathbf{L}}

\newcommand{\boldX}{\mathbf{X}}

\title{Electromagnetic Stekloff eigenvalues: approximation analysis}
\author{Martin Halla}
\email{martin.halla@protonmail.com}
\date{September 2nd, 2019.}

\begin{document}

\begin{abstract}
We continue the work of [Camano, Lackner, Monk, SIAM J.\ Math.\ Anal., Vol.\ 49, No.\ 6, pp.\ 4376-4401 (2017)]
on electromagnetic Stekloff eigenvalues. The authors recognized that in general the eigenvalues due not correspond
to the spectrum of a compact operator and hence proposed a modified eigenvalue problem with the desired properties.
% For the selfadjoint case, the existence and behavior of eigenvalues to the original Stekloff eigenvalue problem
% and the relation between the two eigenvalue problems will be discussed in a companion article [Halla, arxiv- (2019)].

The present article considers the original and the modified electromagnetic Stekloff eigenvalue problem.
We cast the problems as eigenvalue problem for a holomorphic operator function $A(\cdot)$. We construct a
``test function operator function'' $T(\cdot)$ so that $A(\lambda)$ is weakly $T(\lambda)$-coercive for all
suitable $\lambda$, i.e.\ $T(\lambda)^*A(\lambda)$ is a compact perturbation of a coercive operator.
The construction of $T(\cdot)$ relies on a suitable decomposition of the function space into subspaces and an apt
sign change on each subspace.

For the approximation analysis, we apply the framework of T-compatible Galerkin approximations.
% [Halla, arxiv-1908.05029 (2019)] and hence derive convenient convergence results for such.
For the modified problem, we prove that convenient commuting projection operators imply T-compatibility and hence
convergence.
For the original problem, we require the projection operators to satisfy an additional commutator property involving
the tangential trace. The existence and construction of such projection operators remain open questions.
\end{abstract}
\maketitle

\section{Introduction}\label{sec:introduction}
Novel nondestructive evaluation methods based on inverse scattering~\cite{CakoniColton:06} give rise to a multitude of
new eigenvalue problems. Among these are so-called transmission eigenvalue problems~\cite{CakoniColtonHaddar:16}
and Stekloff eigenvalue problems~\cite{CakoniColtonMengMonk:16}. Not all of these eigenvalue problems fall into classes
which are covered in classical literature. Among the important questions on these eigenvalue problems are
\begin{itemize}
 \item Fredholm properties (which imply the discreteness of the spectrum),
 \item the existence of eigenvalues,
 \item properties of the eigenvalues
 \item and reliable computational approximations.
\end{itemize}
The electromagnetic Stekloff eigenvalue problem to find $(\lambda,u)$ so that
\begin{align*}
\curl\curl u-\omega^2\epsilon u &=0 \quad\text{in }\Omega,\\
\nv\times\curl u +\lambda\, \nv\times u\times \nv&=0\quad\text{at }\partial\Omega.
\end{align*}
was considered in the recent publication~\cite{CamanoLacknerMonk:17}. Therein the authors of~\cite{CamanoLacknerMonk:17}
considered the case that $\Omega$ is a ball and the material parameter $\epsilon$ is constant. For this setting they
proved the existence of two infinite sequences of eigenvalues, one converging to zero and one converging to infinity.
Consequently the eigenvalue problem can't be transformed to an eigenvalue problem for a compact operator.
This observation led the authors of~\cite{CamanoLacknerMonk:17} to discard the original eigenvalue problem and
to modify instead the boundary condition to
\begin{align*}
\nv\times\curl u+\lambda S \nv\times u\times \nv
&=0\quad\text{at }\partial\Omega.
\end{align*}
with a suitable operator $S$. The authors of~\cite{CamanoLacknerMonk:17} proved that the modified eigenvalue problem can
indeed be transformed to an eigenvalue problem for a compact operator.\\

In this note we consider the original as well as the modified electromagnetic Stekloff eigenvalue problem.
We formulate the problems as holomorphic operator function eigenvalue problems to find $(\lambda,u)\in \setC\times X$
so that $A(\lambda)u=0$.

We assume reasonable conditions on the material parameters and the domain to analyze the Fredholmness of $A(\lambda)$.
We prove that for the original problem $A(\lambda)$ is Fredholm if and only if $\lambda\in\setC\setminus\{0\}$,
while for the modified problem $A(\lambda)$ is Fredholm for all $\lambda\in\setC$. For our analysis we construct an
operator function $T(\cdot)$ which is bijective at each $\lambda\in\setC\setminus\{0\}$ (respective $\lambda\in\setC$)
so that $T(\lambda)^*A(\lambda)$ is a compact perturbation of a coercive operator. The construction of $T(\cdot)$
involves a decomposition of the function space into subspaces and an apt sign change on each subspace.

We apply the framework of \cite{Halla:19Tcomp} to analyze the convergence of Galerkin approximations. To this end,
we need to construct apt approximations of $T(\cdot)$. We prove for the modified problem, that the existence of
convenient commuting projections imply the existence of such apt approximations of $T(\cdot)$.
For the original problem, we require the projection operators to satisfy an additional commuting property, which
involves the tangential trace, to establish the same result. However, the existence and construction of such projection
operators isn't answered in this article and apt for future research.

We report on the existence and behavior of eigenvalues to the electromagnetic Stekloff eigenvalue problems
in the self adjoint case in the companion article \cite{Halla:19StekloffExist}.\\

The remainder of this article is organized as follows. In Section~\ref{sec:generalsetting} we set our notation and formulate our
assumptions on the domain and the material parameters. We also recall some classic regularity, embedding and decomposition
results which will be essential for our analysis.
In Section~\ref{sec:AFredholm} we introduce the considered electromagnetic Stekloff eigenvalue problem and define the
associated holomorphic operator function $A_X(\cdot)$. We define $T(\cdot)$ and prove that $A_X(\cdot)$ is weakly
$T(\cdot)$-coercive on $\setC\setminus\{0\}$ while $A_X(0)$ is not Fredholm.
In Section~\ref{sec:approximation} we prove that Galerkin approximations which admit uniformly bounded commuting
projections are asymptoticly reliable.
In Section~\ref{sec:modified} we introduce the modified electromagnetic Stekloff eigenvalue problem and define the
associated holomorphic operator function $\At_{\Xt}(\cdot)$. We define $\Tt$ and prove that $\At_{\Xt}(\cdot)$
is weakly $\Tt$-coercive. We introduce a reformulation of the eigenvalue problem by means of an operator function
$\Att^l(\cdot)$, which avoids the explicit appearance of $S$. Likewise we define $\Ttt^l(\cdot)$ and prove that
$\Att^l(\cdot)$ is weakly $\Ttt^l(\cdot)$-coercive.
In Section~\ref{sec:approximationMod} we prove that Galerkin approximations which admit uniformly bounded commuting
projections are asymptoticly reliable. We further discuss the computational implementation of the Galerkin approximations.

\section{General setting}\label{sec:generalsetting}
In this section we set our notation, formulate assumptions on the domain and material parameters, recall necessary
results from different literature.

\subsection{Functional analysis}
For generic Banach spaces $(X, \|\cdot\|_X)$, $(Y, \|\cdot\|_Y)$ denote $L(X,Y)$
the space of all bounded linear operators from $X$ to $Y$ with operator norm
$\|A\|_{L(X,Y)}:=\sup_{u\in X\setminus\{0\}} \|Au\|_Y/\|u\|_X$, $A\in L(X,Y)$.
We further set $L(X):=L(X,X)$. For generic Hilbert spaces
$(X, \spl\cdot,\cdot\spr_X)$, $(Y, \spl\cdot,\cdot\spr_Y)$ and $A\in L(X,Y)$ we denote $A^*\in L(Y,X)$ its adjoint
operator defined through $\spl u,A^*u' \spr_X=\spl Au,u'\spr_Y$ for all $u\in X,u'\in Y$.
We say that an operator $A\in L(X)$ is coercive if $\inf_{u\in X\setminus\{0\}}
|\spl Au,u\spr_X|/\|u\|^2_X>0$. We say that $A\in L(X)$ is weakly coercive, if there
exists a compact operator $K\in L(X)$ so that $A+K$ is coercive. For bijective $T\in L(X)$
we say that $A$ is (weakly) $T$-coercive, if $T^*A$ is (weakly) coercive.
Let $\Lambda\subset\setC$ be open and consider operator functions $A(\cdot), T(\cdot)\colon \Lambda\to L(X)$
so that $T(\lambda)$ is bijective for all $\lambda\in\Lambda$. We call $A(\cdot)$ (weakly) ($T(\cdot)$-)coercive
if $A(\lambda)$ is (weakly) ($T(\lambda)$-)coercive for all $\lambda\in\Lambda$.
We denote the spectrum of $A(\cdot)$ as
$\sigma\big(A(\cdot)\big):=\{\lambda\in\Lambda\colon A(\lambda)\text{ is not bijective}\}$ and
the resolvent set as $\rho\big(A(\cdot)\big):=\Lambda\setminus\sigma\big(A(\cdot)\big)$.
For a subspace $X_n\subset X$ denote $P_n \in L(X,X_n)$ the orthogonal projection.
Consider $A\in L(X)$ to be weakly $T$-coercive. For a sequence $(X_n)_{n\in\setN}$ of finite dimensional subspaces
$X_n\subset X$ with $\lim_{n\in\setN}\|u-P_nu\|_X=0$ for each $u\in X$,
we say that the Galerkin approximation $P_nA|_{X_n}\in L(X_n)$ is $T$-compatible, if
there exists a sequence $(T_n)_{n\in\setN}, T_n\in L(X_n)$ so that
\begin{align}\label{eq:discretenorm}
\|T-T_n\|_n:=\sup_{u_n\in X_n\setminus\{0\}} \|(T-T_n)u_n\|_X /\|u_n\|_X
\end{align}
tends to zero as $n\to\infty$.
Let $A(\cdot)\colon \Lambda\to L(X)$ be weakly $T(\cdot)$-coercive. We say that the Galerkin approximation
$P_nA(\cdot)|_{X_n}\colon\Lambda\to L(X_n)$ is $T(\cdot)$-compatible, if $P_nA(\lambda)|_{X_n}\in L(X_n)$ is
$T(\lambda)$-compatible for each $\lambda\in\Lambda$.

\subsection{Lebesgue and Sobolev spaces}
Let $\Omega\subset\setR^3$ be a bounded path connected Lipschitz domain and $\nv$ the outer unit
normal vector at $\partial\Omega$. We use standard notation for Lebesgue and
Sobolev spaces $L^2(\Omega)$, $L^\infty(\Omega)$, $W^{1,\infty}(\Omega)$, $H^s(\Omega)$ defined on the
domain $\Omega$ and $L^2(\partial\Omega)$, $H^s(\partial\Omega)$ defined on the
boundary $\partial\Omega$. We recall the continuity of the trace operator
$\tr \in L\big(H^s(\Omega), H^{s-1/2}\big)$ for all $s>1/2$.
For a vector space $X$ of scalar valued functions we
denote its bold symbol as space of three-vector valued functions
$\boldX:=X^3=X\times X\times X$, e.g.\ $\boldL^2(\Omega)$, $\boldH^s(\Omega)$,
$\boldL^2(\partial\Omega)$, $\boldH^s(\partial\Omega)$. For $\boldL^2(\partial\Omega)$
or a subspace, e.g.\ $\boldH^s(\partial\Omega), s>0$, the subscript $t$ denotes
the subspace of tangential fields. In particular $\boldL^2_t(\partial\Omega)=
\{u\in \boldL^2(\partial\Omega)\colon \nv\cdot u=0\}$ and $\boldH^s_t(\partial\Omega)=
\{u\in \boldH^s(\partial\Omega)\colon \nv\cdot u=0\}$. Let further $H^1_0(\Omega)$
be the subspace of $H^1(\Omega)$ of all functions with vanishing Dirichlet trace,
$H^1_*(\Omega)$ be the subspace of $H^1(\Omega)$ of all functions with vanishing mean,
i.e.\ $\spl u,1\spr_{L^2(\Omega)}=0$ and $H^1_*(\partial\Omega)$ be the subspace of
$H^1(\partial\Omega)$ of all functions with vanishing mean $\spl u,1 \spr_{L^2(\partial\Omega)}=0$.

\subsection{Additional function spaces}
Denote $\partial_{x_i} u$ the partial derivative of a function $u$ with respect
to the variable $x_i$. Let
\begin{align*}
\nabla u&:=(\partial_{x_1}u,\partial_{x_2}u,\partial_{x_3}u)^\top,\\
\div (u_1,u_2,u_3)^\top&:=\partial_{x_1}u_1+\partial_{x_2}u_2+\partial_{x_3}u_3,\\
\curl (u_1,u_2,u_3)^\top&:=(-\partial_{x_2}u_1+\partial_{x_1}u_3,
\partial_{x_3}u_1-\partial_{x_1}u_3,-\partial_{x_2}u_1+\partial_{x_1}u_2)^\top.
\end{align*}
For a bounded Lipschitz domain $\Omega$ let $\nabla_\partial, \div_\partial$ and $\curl_\partial=\nu\times\nabla_\partial$
be the respective differential operators for functions defined on $\partial\Omega$.
We recall that for $u\in\boldL^2(\Omega)$ with $\curl u\in\boldL^2(\Omega)$ the
tangential trace $\tr_{\nv\times}u\in \boldH^{-1/2}(\div_\partial;\partial\Omega):=\{u \in
\boldH^{-1/2}(\partial\Omega)\colon \div_\partial u\in H^{-1/2}(\partial\Omega)\}$,
$\|u\|^2_{\boldH^{-1/2}(\div_\partial;\partial\Omega)}:=\|u\|^2_{\boldH^{-1/2}(\partial\Omega)}
+\|\div_\partial u\|^2_{H^{-1/2}(\partial\Omega)}$ is well defined and
$\|\tr_{\nv\times}u\|_{\boldH^{-1/2}(\div_\partial;\partial\Omega)}^2$
is bounded by a constant times $\|u\|^2_{\boldL^2(\Omega)}+\|\curl u\|^2_{\boldL^2(\Omega)}$.
Likewise for $u\in\boldL^2(\Omega)$ with $\div u\in L^2(\Omega)$
the normal trace $\tr_{\nv\cdot}u\in H^{-1/2}(\partial\Omega)$ is well defined
and $\|\tr_{\nv\cdot}u\|_{H^{-1/2}(\partial\Omega)}^2$ is
bounded by a constant times $\|u\|_{\boldL^2(\Omega)}^2+\|\div u\|_{L^2(\Omega)}^2$.
For $\dd\in\{\curl,\div,\tr_{\nv\times},\tr_{\nv\cdot}\}$ let
\begin{subequations}
\begin{align}
L^2(\dd):=\left\{\begin{array}{ll}
\boldL^2(\Omega),&\dd=\curl,\\
L^2(\Omega),&\dd=\div,\\
\boldL^2_t(\partial\Omega),&\dd=\tr_{\nv\times},\\
L^2(\partial\Omega),&\dd=\tr_{\nv\cdot}
\end{array}\right..
\end{align}
Let
\begin{align}
\begin{aligned}
H(\dd;\Omega)&:=\{u\in L^2(\Omega)\colon \dd u\in L^2(\dd)\},\\
\spl u,u'\spr_{H(\dd;\Omega)}&:=\spl u,u'\spr_{L^2(\Omega)}
+\spl \dd u,\dd u'\spr_{L^2(\dd)},
\end{aligned}
\end{align}
\begin{align}
H(\dd^0;\Omega)&:=\{u\in H(\dd;\Omega)\colon \dd u=0\}.
\end{align}
Also for
\begin{align*}
\dd_1,\dd_2,\dd_3,\dd_4\in
\{&\curl,\div,\tr_{\nv\times},\tr_{\nv\cdot},\curl^0,\div^0,\tr_{\nv\times}^0,\tr_{\nv\cdot}^0\}
\end{align*}
let
\begin{align}
\begin{aligned}
H(\dd_1,\dd_2;\Omega)&:=H(\dd_1;\Omega)\cap H(\dd_2;\Omega),\\
\spl u,u'\spr_{H(\dd_1,\dd_2;\Omega)}&:=\spl u,u'\spr_{L^2(\Omega)}
+\spl \dd_1 u,\dd_1 u'\spr_{L^2(\dd_1)}+\spl \dd_2 u,\dd_2 u'\spr_{L^2(\dd_2)},
\end{aligned}
\end{align}
\begin{align}
\begin{aligned}
H(\dd_1,\dd_2,\dd_3;\Omega)&:=H(\dd_1;\Omega)\cap H(\dd_2;\Omega)\cap H(\dd_3;\Omega),\\
\spl u,u'\spr_{H(\dd_1,\dd_2,\dd_3;\Omega)}&:=\spl u,u'\spr_{L^2(\Omega)}
+\spl \dd_1 u,\dd_1 u'\spr_{L^2(\dd_1)}+\spl \dd_2 u,\dd_2 u'\spr_{L^2(\dd_2)}\\
&+\spl \dd_3 u,\dd_3 u'\spr_{L^2(\dd_3)},
\end{aligned}
\end{align}
and
\begin{align}
\begin{aligned}
H(\dd_1,\dd_2,\dd_3,\dd_4;\Omega)&:=H(\dd_1;\Omega)\cap H(\dd_2;\Omega)
\cap H(\dd_3;\Omega)\cap H(\dd_4;\Omega),\\
\spl u,u'\spr_{H(\dd_1,\dd_2,\dd_3,\dd_4;\Omega)}&:=\spl u,u'\spr_{L^2(\Omega)}
+\spl \dd_1 u,\dd_1 u'\spr_{L^2(\dd_1)}+\spl \dd_2 u,\dd_2 u'\spr_{L^2(\dd_2)}\\
&+\spl \dd_3 u,\dd_3 u'\spr_{L^2(\dd_3)}+\spl \dd_4 u,\dd_4 u'\spr_{L^2(\dd_4)}.
\end{aligned}
\end{align}
\end{subequations}

\subsection{Assumptions on the domain and material parameters}
\begin{assumption}[Assumption on $\epsilon$]\label{ass:eps}
Let $\epsilon\in \big(L^\infty(\Omega)\big)^{3x3}$ be so that there exist $c_\epsilon>0$ with
\begin{align}
c_\epsilon |\xi|^2 &\leq \Re (\xi^H \epsilon(x) \xi) \quad\text{and}\quad 0 \leq \Im (\xi^H \epsilon(x) \xi)
\end{align}
for all $x\in\Omega$ and all $\xi\in\setC^3$.
\end{assumption}

\begin{assumption}[Assumption on $\mu$]\label{ass:mu}
Let $\mu^{-1}\in \big(L^\infty(\Omega)\big)^{3x3}$ be so that there exist $c_\mu>0$ with
\begin{align}
c_\mu |\xi|^2 &\leq \Re (\xi^H \mu^{-1}(x) \xi) \quad\text{and}\quad
0 \leq -\Im (\xi^H \mu^{-1}(x) \xi)
\end{align}
for all $x\in\Omega$ and all $\xi\in\setC^3$.
\end{assumption}

\begin{assumption}[Assumption on $\Omega$]\label{ass:Domain}
Let $\Omega\subset\setR^3$ be a bounded path connected Lipschitz domain so that there
exists $\delta>0$ and the following shift theorem holds on $\Omega$:
Let $f\in L^2(\Omega)$, $g\in H^{1/2}(\partial\Omega)$ with
$\spl g,1\spr_{L^2(\partial\Omega)}=0$ and $w\in H^1_*(\Omega)$ be the solution to
\begin{subequations}
\begin{align}
-\Delta w &= f\quad\text{ in }\Omega,\\
n\cdot\nabla w&= g\quad\text{ at }\partial\Omega.
\end{align}
\end{subequations}
Then the linear map $(f,g)\mapsto w\colon L^2(\Omega)\times
H^{1/2}(\partial\Omega)\to H^{3/2+\delta}(\Omega)$ is well defined and continuous.
\end{assumption}
The above assumption holds e.g.\ for smooth domains and Lipschitz polyhedral~\cite[Corollary~23.5]{Dauge:88}.

\begin{assumption}[Assumption on $\Omega, \epsilon$ and $\mu^{-1}$]\label{ass:UCP}
Let $\epsilon, \mu^{-1}$ and $\Omega$ be so that a unique continuation principle holds, i.e.\ if $u\in H(\curl;\Omega)$
solves
\begin{subequations}
\begin{align}
\curl\mu^{-1}\curl u -\omega^2\epsilon u&=0\quad\text{in }\Omega,\\
\tr_{\nv\times} u&=0\quad\text{at }\partial\Omega,\\
\tr_{\nv\times} \mu^{-1}\curl u&=0\quad\text{at }\partial\Omega,
\end{align}
\end{subequations}
then $u=0$.
\end{assumption}
To our knowledge the most general todays available result on the unique
continuation principle for Maxwells equations is the one of Ball, Capdeboscq and
Tsering-Xiao~\cite{BallCapdeboscqTsering-Xiao:12}. It essentially requires $\epsilon$ and $\mu^{-1}$ to be piece-wise
$W^{1,\infty}$.

\subsection{Trace regularities and compact embeddings}
We recall some classical results on traces and embeddings, which will be essential for our analysis.
We recall from Costabel~\cite{Costabel:90}:
\begin{subequations}\label{eq:CostabelTrace}
\begin{align}
\tr_{\nv\cdot}\in  L\big(H(\curl,\div,\tr_{\nv\times};\Omega), L^2(\partial\Omega)\big),\\
\tr_{\nv\times}\in L\big(H(\curl,\div,\tr_{\nv\cdot};\Omega), \boldL^2_t(\partial\Omega)\big).
\end{align}
\end{subequations}
and
\begin{align}\label{eq:CostabelDomain}
\begin{aligned}
&\text{The embedding from }H(\curl,\div,\tr_{\nv\times};\Omega)\text{ to } \boldL^2(\Omega)\text{ is compact}.
\end{aligned}
\end{align}
We deduce from Amrouche, Bernardi, Dauge and Girault~\cite[Proposition~3.7]{AmroucheBernardiDaugeGirault:98}:
\begin{align}\label{eq:Vtraceregularity}
\begin{aligned}
&\text{If } \Omega \text{ suffices Assumption~\ref{ass:Domain}, then}
\tr_{\nv\times} \in L\big( H(\curl,\div,\tr_{\nv\cdot}^0;\Omega),\boldL^2_t(\partial\Omega) \big)\\
&\text{is compact}.
\end{aligned}
\end{align}

\subsection{Helmholtz decomposition on the boundary}
We recall from Buffa, Costabel and Sheen~\cite[Theorem 5.5]{BuffaCostabelSheen:02}:
\begin{align}
\boldL^2_t(\partial\Omega)=\nabla_\partial H^1(\partial\Omega) \oplus^\bot
\curl_\partial H^1(\partial\Omega).
\end{align}
and denote the respective orthogonal projections by
\begin{align}
P_{\nabla_\partial}\colon \boldL^2_t(\partial\Omega) \to \nabla_\partial H^1(\partial\Omega),\qquad
P_{\nabla_\partial^\top}\colon \boldL^2_t(\partial\Omega) \to \curl_\partial H^1(\partial\Omega).
\end{align}
Recall $\div_\partial \tr_{\nv\times} \in L\big(H(\curl;\Omega), H^{-1/2}(\partial\Omega)\big)$.
So for $u\in H(\curl;\Omega)$ let $z$ be the solution to find $z\in H^1_*(\partial\Omega)$ so that
\begin{align}
\spl \nabla_\partial z, \nabla_\partial z' \spr_{\boldL_t^2(\partial\Omega)}
= -\spl \div_\partial \tr_{\nv\times} u, z' \spr_{H^{-1}(\partial\Omega)\times H^1(\partial\Omega)}
\end{align}
for all $z'\in H^1_*(\partial\Omega)$ and set
\begin{align}\label{eq:DefS}
Su:=\nabla_\partial z.
\end{align}
From the construction of $S$ it follows $S\in L\big(H(\curl;\Omega),\boldL_t^2(\partial\Omega)\big)$ and further
\begin{align}
Su=P_{\nabla_\partial}\tr_{\nv\times} u
\end{align}
for $u\in H(\curl,\tr_{\nv\times};\Omega)$.

\section{Weak $T(\cdot)$-coercivity of the Stekloff operator function}\label{sec:AFredholm}
First we introduce the electromagnetic Stekloff eigenvalue problem as holomorphic operator function eigenvalue problem.
In Theorem~\ref{thm:VW} we report an apt decomposition of the respective Hilbert space into three subspaces.
Next we introduce in~\eqref{eq:T} an operator function $T(\cdot)$ as an apt sign change on the subspaces.
In Theorem~\ref{thm:AwTc} we report the weak $T(\cdot)$-coercivity of the Stekloff operator function on
$\setC\setminus\{0\}$. In Corollary~\ref{cor:ressetnonempty} we deduce convenient properties of the spectrum in
$\setC\setminus\{0\}$. In Corollary~\ref{cor:AzNotFred} we report that $\lambda=0$ constitutes the essential spectrum.
However, these two Corollaries make no statement on the existence of eigenvalues. We report in a companion article
\cite{Halla:19StekloffExist}
the existence and behavior of eigenvalues for purely real, symmetric $\mu$ and $\epsilon$, i.e.\ in the selfadjoint case.\\

Let $\omega>0$ be fixed. For $\lambda\in\setC$ let $A(\lambda)\in L\big(H(\curl,\tr_{\nv\times};\Omega)\big)$ be defined
through
\begin{align}
\begin{aligned}
\spl A(\lambda)u,u'&\spr_{H(\curl,\tr_{\nv\times};\Omega)}:=
\spl \mu^{-1} \curl u,\curl u'\spr_{\boldL^2(\Omega)}
-\omega^2\spl\epsilon u,u'\spr_{\boldL^2(\Omega)}\\
&-\lambda \spl \tr_{\nv\times} u,\tr_{\nv\times} u'\spr_{\boldL^2_t(\partial\Omega)}
\quad\text{for all }u,u'\in H(\curl,\tr_{\nv\times};\Omega).
\end{aligned}
\end{align}
The electromagnetic Stekloff eigenvalue problem which we investigate in this section is to
\begin{align}\label{eq:EVP}
\text{find}\quad (\lambda,u)\in \setC\times H(\curl,\tr_{\nv\times};\Omega)
\setminus\{0\}\quad\text{so that}\quad A(\lambda)u=0.
\end{align}
We note that the sign of $\lambda$ herein is reversed compared to~\cite{CamanoLacknerMonk:17}.
To analyze the operator $A(\lambda)$ we introduce the following subspaces of $H(\curl,\tr_{\nv\times};\Omega)$:
\begin{subequations}
\begin{align}
V&:=H(\curl,\div^0,\tr_{\nv\times},\tr_{\nv\cdot}^0;\Omega),\\
W_1&:=H(\curl^0,\div^0,\tr_{\nv\times};\Omega),\\
W_2&:=\nabla H^1_0 \subset H(\curl^0,\tr_{\nv\times}^0;\Omega).
\end{align}
\end{subequations}
\begin{theorem}\label{thm:VW}
It holds
\begin{align}
H(\curl,\tr_{\nv\times};\Omega)=(V\oplus W_1)\oplus^{\bot_{H(\curl,\tr_{\nv\times};\Omega)}} W_2
\end{align}
in the following sense. There exist projections $P_V, P_{W_1}, P_{W_2} \in L\big(H(\curl,\tr_{\nv\times};\Omega)\big)$
with $\ran P_V=V, \ran P_{W_1}=W_1, \ran P_{W_2}=W_2$,
$W_1, W_2 \subset \ker P_V$, $V, W_2 \subset \ker P_{W_1}$, $V, W_1 \subset \ker P_{W_2}$
and $u=P_vu+P_{W_1}u+P_{W_2}u$ for each $u\in H(\curl,\tr_{\nv\times};\Omega)$. Thus, the norm induced by
\begin{align}\label{eq:scptildeX}
\begin{aligned}
\spl u,u'\spr_X&:=\spl P_Vu,P_Vu'\spr_{H(\curl,\tr_{\nv\times};\Omega)}
+\spl P_{W_1}u,P_{W_1}u'\spr_{H(\curl,\tr_{\nv\times};\Omega)}\\
&+\spl P_{W_2}u,P_{W_2}u'\spr_{H(\curl,\tr_{\nv\times};\Omega)},
\quad u,u'\in H(\curl,\tr_{\nv\times};\Omega)
\end{aligned}
\end{align}
is equivalent to $\|\cdot\|_{H(\curl,\tr_{\nv\times};\Omega)}$.
\end{theorem}
\begin{proof}
\textit{1. Step:}\quad
Let $P_{W_2}$ be the $H(\curl,\tr_{\nv\times};\Omega)$-orthogonal projection onto $W_2$.
Hence $P_{W_2}\in L\big(H(\curl,\tr_{\nv\times};\Omega)\big)$ is a projection with range $W_2$ and kernel
\begin{align*}
W_2^{\bot_{H(\curl,\tr_{\nv\times};\Omega)}}=H(\curl,\div^0,\tr_{\nv\times};\Omega)\supset V, W_1.
\end{align*}

\textit{2a. Step:}\quad
Let $u\in H(\curl,\tr_{\nv\times};\Omega)$. Note that due to $\div (u-P_{W_2}u)=0$ and \eqref{eq:CostabelTrace} it hold
$\tr_{\nv\cdot}(u-P_{W_2}u) \in L^2(\partial\Omega)$ and
$\spl \tr_{\nv\cdot}(u-P_{W_2}u), 1\spr_{L^2(\partial\Omega)}=0$.
Let $w_*\in H^1_*(\Omega)$ be the unique solution to
\begin{align*}
-\Delta w_*=0 \quad\text{in }\Omega, \qquad \nv\cdot \nabla w_*= \tr_{\nv\cdot}(u-P_{W_2}u)
\quad\text{at }\partial\Omega.
\end{align*}
Let $P_{W_1}u:=\nabla w_*$. By construction of $P_{W_1}$ and due to \eqref{eq:CostabelTrace} it hold
$\ran P_{W_1}\subset W_1$ and $P_{W_1}\in L\big(H(\curl,\tr_{\nv\times};\Omega)\big)$.
Let $u\in W_1$. Then $P_{W_2}u=0$ and hence $P_{W_1}u=u$. Thus $P_{W_1}$ is a projection and $\ran P_{W_1}=W_1$.

\textit{2b. Step:}\quad
If $u\in W_2$ then $u-P_{W_2}u=0$, further $\tr_{\nv\cdot}(u-P_{W_2}u)=0$ and thus $P_{W_1}u=0$.
Hence $W_2\subset \ker P_{W_1}$.
If $u\in V$ then $P_{W_2}u=0$, further $\tr_{\nv\cdot}(u-P_{W_2}u)=\tr_{\nv\cdot}u=0$ and thus $P_{W_1}u=0$.
Hence $V\subset \ker P_{W_1}$.

\textit{3. Step:}\quad
Let $u\in H(\curl,\tr_{\nv\times};\Omega)$ and $P_Vu:=u-P_{W_1}u-P_{W_2}u$. It follow
$P_V\in L\big(H(\curl,\tr_{\nv\times};\Omega)\big)$, $P_Vu\in V$ and $P_VP_Vu=P_Vu$, i.e.\ $P_V$ is a bounded projection.
If $u\in V$ then $P_{W_1}u=P_{W_2}u=0$ and thus $P_Vu=u$. Hence $\ran P_V=V$. It follow further $W_1,W_2 \subset \ker P_V$.

\textit{4. Step:}\quad
By means of the triangle
inequality and a Young inequality it holds.
\begin{align*}
\|u\|&_{{H(\curl,\tr_{\nv\times};\Omega)}}^2
=\|P_Vu+P_{W_1}u+P_{W_2}u\|_{{H(\curl,\tr_{\nv\times};\Omega)}}^2\\
&\leq 3\big(\|P_Vu\|_{{H(\curl,\tr_{\nv\times};\Omega)}}^2
+\|P_{W_1}u\|_{{H(\curl,\tr_{\nv\times};\Omega)}}^2
+\|P_{W_2}u\|_{{H(\curl,\tr_{\nv\times};\Omega)}}^2\big)\\
&=3\|u\|_X^2.
\end{align*}
On the other hand due to the boundedness of the projections
\begin{align*}
\|u\|_X^2
&=\|P_Vu\|_{{H(\curl,\tr_{\nv\times};\Omega)}}^2
+\|P_{W_1}u\|_{{H(\curl,\tr_{\nv\times};\Omega)}}^2
+\|P_{W_2}u\|_{{H(\curl,\tr_{\nv\times};\Omega)}}^2\\
&\leq \big(\|P_V\|_{L(H(\curl,\tr_{\nv\times};\Omega))}^2
+\|P_{W_1}\|_{L(H(\curl,\tr_{\nv\times};\Omega))}^2\\
&+\|P_{W_2}\|_{L(H(\curl,\tr_{\nv\times};\Omega))}^2\big)
\|u\|_{{H(\curl,\tr_{\nv\times};\Omega)}}^2.
\end{align*}
\end{proof}

Let us look at $A(\lambda)$ in light of this substructure of ${H(\curl,\tr_{\nv\times};\Omega)}$. To this end we
consider the space
\begin{align}\label{eq:X}
X:=H(\curl,\tr_{\nv\times};\Omega), \qquad
\spl\cdot,\cdot\spr_X\quad\text{as defined in~\eqref{eq:scptildeX}}.
\end{align}
It follows that $P_V, P_{W_1}$ and $P_{W_1}$ are even orthogonal projections in
$X$. Let further $A_X(\cdot), A_c, A_\epsilon, A_{l^2}, A_{\tr} \in
L(X)$ be defined through
\begin{subequations}
\begin{align}
\label{eq:DefAX}
\spl A_X(\lambda) u,u'\spr_X&:=\spl A(\lambda)u,u'\spr_{H(\curl,\tr_{\nv\times};\Omega)}
\quad\text{for all }u,u'\in X,\lambda\in\setC\\
\spl A_c u,u'\spr_X&:=\spl \mu^{-1}\curl u,\curl u'\spr_{\boldL^2(\Omega)}
\quad\text{for all }u,u'\in X,\\
\spl A_\epsilon u,u'\spr_X&:=\spl \epsilon u,u'\spr_{\boldL^2(\Omega)}
\quad\text{for all }u,u'\in X,\\
\spl A_{l^2} u,u'\spr_X&:=\spl u,u'\spr_{\boldL^2(\Omega)}
\quad\text{for all }u,u'\in X,\\
\spl A_{\tr} u,u'\spr_X&:=\spl \tr_{\nv\times}u,\tr_{\nv\times}u'\spr_{\boldL^2_t(\partial\Omega)}
\quad\text{for all }u,u'\in X.
\end{align}
\end{subequations}
We deduce from the definitions of $V, W_1$ and $W_2$ that
\begin{align}\label{eq:AVWW}
\begin{aligned}
A_X(\lambda)&=
(P_V+P_{W_1}+P_{W_2})(A_c-\omega^2 A_\epsilon-\lambda A_{\tr})(P_V+P_{W_1}+P_{W_2})\\
&=P_VA_cP_V
-\omega^2 (P_V+P_{W_1}+P_{W_2}) A_\epsilon (P_V+P_{W_1}+P_{W_2})\\
&-\lambda (P_V+P_{W_1})A_{\tr} (P_V+P_{W_1})\\
&={\color{blue}P_VA_cP_V -\omega^2 P_{W_2}A_\epsilon P_{W_2} -\lambda P_{W_1}A_{\tr} P_{W_1}}\\
&-\omega^2 (P_VA_\epsilon P_V +P_{W_1}A_\epsilon P_{W_1})\\
&-\lambda \big(P_VA_{\tr} P_V +P_VA_{\tr} P_{W_1} +P_{W_1}A_{\tr} P_V).
\end{aligned}
\end{align}
If we identify $X\sim V\times W_1\times W_2$ and
$X \ni u\sim (v,w_1,w_2)\in V\times W_1\times W_2$,
we can identify $A_X(\lambda)$ with the block operator
\begin{align}\label{eq:Ablock}
\bpm
{\color{blue}P_VA_c|_V} -P_V(\omega^2 A_\epsilon +\lambda A_{\tr}) |_V & -P_V (\omega^2 A_\epsilon+ \lambda A_{\tr}) |_{W_1}& -\omega^2 P_V A_\epsilon |_{W_2}\\
- P_{W_1} (\omega^2 A_\epsilon+\lambda A_{\tr}) |_V&-\omega^2 P_{W_1} A_\epsilon|_{W_1} {\color{blue}-\lambda P_{W_1} A_{\tr}|_{W_1}}& -\omega^2 P_{W_1} A_\epsilon |_{W_2}\\
-\omega^2 P_{W_2} A_\epsilon |_V & -\omega^2 P_{W_2} A_\epsilon |_{W_1} &{\color{blue}-\omega^2 P_{W_2} A_\epsilon |_{W_2}}
\epm.
\end{align}
We color highlighted in \eqref{eq:AVWW} and \eqref{eq:Ablock} the operators which are not compact. This leads us to
define a test function operator function in the following way. Let
\begin{align}\label{eq:T}
T(\lambda)&:=P_V-\ol{\lambda}^{-1}P_{W_1}-\omega^{-2}P_{W_2}, \quad\lambda\in\setC\setminus\{0\}.
\end{align}
Obviously $T(\lambda)\in L(X)$ is bijective with $T(\lambda)^{-1}=P_V-\ol{\lambda} P_{W_1}-\omega^2 P_{W_2}$ for
$\lambda\in\setC\setminus\{0\}$.

\begin{theorem}\label{thm:AwTc}
Let $\epsilon$ suffice Assumption~\ref{ass:eps}, $\mu$ suffice Assumption~\ref{ass:mu}
and $\Omega$ suffice Assumption~\ref{ass:Domain}. Thence $A_X(\cdot)\colon \setC\setminus\{0\}\to L(X)$
is weakly $T(\cdot)$-coercive.
\end{theorem}
\begin{proof}
Let $\lambda\in\setC\setminus\{0\}$. Let
\begin{align*}
A_1&:=P_VA_cP_V
+P_V A_{l^2} P_V
+P_V A_{\tr} P_V\\
&-\lambda P_{W_1} A_{l^2} P_{W_1}
-\lambda P_{W_1} A_{\tr} P_{W_1}
-\omega^2 P_{W_2} A_\epsilon P_{W_2}
\end{align*}
and
\begin{align*}
A_2&:=
-\omega^2\big(P_V A_\epsilon P_V
+P_{W_1} A_\epsilon P_{W_1}
+P_V A_\epsilon P_{W_1}
+P_{W_1} A_\epsilon P_V\\
&+P_V A_\epsilon P_{W_2}
+P_{W_2} A_\epsilon P_V
+P_{W_1} A_\epsilon P_{W_2}
+P_{W_2} A_\epsilon P_{W_1}
\big)\\
&-P_V A_{l^2} P_V -(1+\lambda)P_V A_{\tr} P_V \\
&+\lambda P_{W_1} A_{l^2} P_{W_1}-\lambda(P_V A_{\tr} P_{W_1}
+P_{W_1} A_{\tr} P_V\big).
\end{align*}
so that $A_X(\lambda)=A_1+A_2$. Operator $A_2$ is compact due to \eqref{eq:CostabelDomain} and
\eqref{eq:Vtraceregularity} and hence so is $T^*A_2$. It is straight forward to see
\begin{align*}
\Re(\spl A_1u,T(\lambda)u \spr_X) \geq \min(1,c_\epsilon,c_\mu) \|u\|_X^2,
\end{align*}
i.e.\ $T(\lambda)^*A_1$ is coercive.
\end{proof}

We remark that the naming of the (sub)spaces as $X,V,W_1,W_2$ follows
Buffa~\cite{Buffa:05} while the naming of the ``test function operator'' as
$T(\lambda)$ follows e.g.\ Bonnet-Ben Dhia, Ciarlet and Zw\"olf~\cite{BonnetBDCiarletZwoelf:10}.

\begin{corollary}\label{cor:ressetnonempty}
Let Assumptions~\ref{ass:eps}, \ref{ass:mu}, \ref{ass:Domain} and \ref{ass:UCP} hold true.
Then $A_X(\lambda)$ is bijective for all $\lambda\in\setC$ with $\Im(\lambda)<0$.
Hence the spectrum of $A_X(\cdot)$ in $\setC\setminus\{0\}$ consists of an at most countable set of eigenvalues
with finite algebraic multiplicity which have no accumulation point in $\setC\setminus\{0\}$.
\end{corollary}
\begin{proof}
Let $\lambda\in\setC$ with $\Im(\lambda)<0$ and $u \in X$ be so that $A_X(\lambda)u=0$. It follows
\begin{align*}
0=-\Im(\spl A_X(\lambda)u,u)\spr_X)
\geq-\Im(\lambda)\|\tr_{\nv\times}u\|_{\boldL^2_t(\partial\Omega)}^2
\end{align*}
and together with Assumption~\ref{ass:UCP} it follows further $u=0$, i.e.\ $A_X(\lambda)$ is injective.
From Theorem~\ref{thm:AwTc} we know that $A_X(\lambda)$ is Fredholm with index zero for all
$\lambda\in\setC\setminus\{0\}$ and hence $A_X(\lambda)$ is bijective, if $\Im(\lambda)<0$.

Further $A_X(\cdot)$ is holomorphic since it is even an affine function. The resolvent set of
$A_X(\cdot)\colon \setC\setminus\{0\}\to L(X)$ is non-empty.
The result on the spectrum in $\setC\setminus\{0\}$ is a classical result on
holomorphic Fredholm operator functions, see e.g.~\cite[Proposition~A.8.4]{KozlovMazya:99}.
\end{proof}

\begin{corollary}\label{cor:AzNotFred}
Let $\epsilon$ suffice Assumptions~\ref{ass:eps}. Then $A_X(0)$ is not Fredholm.
\end{corollary}
\begin{proof}
We construct a singular sequence $(w_{1,n}\in W_1)_{n\in\setN}$ for $A(0)$, i.e.\ $\|w_{1,n}\|_X=1$ for each $n\in\setN$,
$(w_{1,n})_{n\in\setN}$ admits no converging subsequence and $\lim_{n\in\setN} A(0)w_{1,n}=0$.

To this end let $(f_n\in L^2(\partial\Omega)\setminus\{0\})_{n\in\setN}$ be a sequence which admits no converging
subsequence and which converges to $f\in H^{-1/2}(\partial\Omega)\setminus L^2(\partial\Omega)$ in
$H^{-1/2}(\partial\Omega)$ so that $\|f_n\|_{L^2(\partial\Omega)}\to+\infty$ as $n\to+\infty$.
Let $\tilde w_{1,n}\in H^1_*(\Omega)$ be the solution to
\begin{align*}
-\Delta \tilde w_{1,n} &=0\quad\text{in }\Omega,\\
\nv\cdot\nabla \tilde w_{1,n} &=f_n\quad\text{at }\partial\Omega.
\end{align*}
The volume part of the norm $\|\nabla \tilde w_{1,n}\|_{\boldL^2(\Omega)}$ can
be uniformly bounded by
\begin{align*}
\sup_{n\in\setN}\|f_n\|_{H^{-1/2}(\partial\Omega)}.
\end{align*}
Due to~\eqref{eq:CostabelTrace} we know $\|\tr_{\nv\times}\nabla\tilde w_{1,n}\|_{\boldL_t^2(\partial\Omega)}<+\infty$
and there exists $C>0$ independent of $\nabla\tilde w_{1,n}$ so that
\begin{align*}
\|f_n\|_{L^2(\partial\Omega)}
=\|\tr_{\nv\cdot}\nabla\tilde w_{1,n}\|_{L^2(\partial\Omega)}
\leq C(\|\nabla \tilde w_{1,n}\|_{\boldL^2(\Omega)}
+\|\tr_{\nv\times}\nabla\tilde w_{1,n}\|_{\boldL_t^2(\partial\Omega)}).
\end{align*}
It follows $\|\tr_{\nv\times}\nabla\tilde w_{1,n}\|_{\boldL_t^2(\partial\Omega)}
\to+\infty$ as $n\to+\infty$. Hence
\begin{align*}
\|A_X(0)\nabla\tilde w_{1,n}\|_X\leq
\sqrt{3} \|\epsilon\|_{(L^\infty(\Omega))^{3\times 3}}
\|\nabla\tilde w_{1,n}\|_{\boldL^2(\Omega)}.
\end{align*}
Let $w_{1,n}:=\nabla\tilde w_{1,n}/\|\nabla\tilde w_{1,n}\|_X$.
It follows $\|w_{1,n}\|_X=1$ and $A_X(0)w_{1,n}\to0$ as $n\to+\infty$.
The existence of a converging subsequence of $(w_{1,n}\in W_1)_{n\in\setN}$ would imply that
$(f_n\in L^2(\partial\Omega))_{n\in\setN}$ admits a converging subsequence, which is a contradiction.
Hence $(w_{1,n})_{n\in\setN}$ is indeed a singular sequence for $A_X(0)$.
\end{proof}

\section{Compatible approximation of the Stekloff eigenvalue problem}\label{sec:approximation}
In this section we discuss Galerkin approximations of~\eqref{eq:EVP}. In addition to the basic Assumption~\ref{ass:XnI}
we embrace in Assumption~\ref{ass:XnII} the existence of uniformly bounded commuting projections like in
\cite{ArnoldFalkWinther:10}. Since we work with the space $H(\curl,\tr_{\nv\times};\Omega)$ rather than $H(\curl;\Omega)$,
our assumption involves an additional projection on $\boldL^2_t(\partial\Omega)$ compared to~\cite{ArnoldFalkWinther:10}.
We report in Corollary~\ref{cor:Tncompatible} that for Galerkin approximations which satisfy these two assumptions,
we can construct a sequence of operator functions $T_n(\cdot)\colon \setC\setminus\{0\}\to L(X_n)$ which converges to
$T(\cdot)$ in discrete norm~\eqref{eq:discretenorm} at each $\lambda\in\setC\setminus\{0\}$. The prove is based on
Lemma~\ref{lem:QuasiBestAppr} and Lemma~\ref{lem:discnormP} and applies techniques as outlined in~\cite{ArnoldFalkWinther:10}.
Consequently we report in Theorem~\ref{thm:SpecAppr} that the abstract framework of~\cite{Halla:19Tcomp}
(which is based on the exhaustive works of Karma~\cite{Karma:96a}, \cite{Karma:96b}) is applicable.
However, the existence and possible construction of such projection operators remain open questions!\\

Consider the following basic assumption.
\begin{assumption}\label{ass:XnI}
Let $(X_n)_{n\in\setN}$ be so that $X_n\subset X$ and $\dim X_n<\infty$ for each $n\in\setN$, and
\begin{align}
\lim_{n\in\setN}\inf_{u'\in X_n}\|u-u'\|_X=0\text{ for each }u\in X.
\end{align}
\end{assumption}

Consider the following additional assumption.
\begin{assumption}\label{ass:XnII}
There exists $(\pi^X_n)_{n\in\setN}$ so that
\begin{subequations}
\begin{align}
\pi^X_n\in L\big(\boldL^2(\Omega)\big) \text{ is a projector with }X_n = \ran \pi^X_n,\\
\sup_{n\in\setN} \|\pi^X_n\|_{L(\boldL^2(\Omega))} < +\infty.
\end{align}
\end{subequations}
Let $Y:=\boldL^2(\Omega)$ and $Z:=\boldL^2_t(\partial\Omega)$.
There exist sequences $(Y_n,Z_n,\pi^Y_n,\pi^Z_n,)_{n\in\setN}$ so that for each $H\in\{Y,Z\}$ it hold
\begin{subequations}
\begin{align}
H_n\subset H, \qquad \lim_{n\in\setN}\inf_{u'\in H_n}\|u-u'\|_H=0,\\
\pi^H_n\in L(H) \text{ is a projector with }H_n \subset \ran \pi^H_n,\\
\sup_{n\in\setN} \|\pi^H_n\|_{L(H)} < +\infty.
\end{align}
\end{subequations}
Denote $E\in L\big(X,\boldL^2(\Omega)\big)$ the embedding operator and set
\begin{align}
\pi_n:=\pi^X_nE.
\end{align}
Further let
\begin{align}
\curl\circ\pi_n u=\pi^Y_n\circ\curl u \qquad\text{and}\qquad
\tr_{\nv\times} \circ \pi_n u=\pi^Z_n \circ \tr_{\nv\times} u
\end{align}
for each $u\in X$.
\end{assumption}

\begin{lemma}\label{lem:QuasiBestAppr}
Let Assumptions~\ref{ass:XnI} and \ref{ass:XnI} hold true. Then the projections $\pi^X_n, \pi^Y_n$ and $\pi^Z_n$
converge point-wise to the identity in $\boldL^2(\Omega), \boldL^2(\Omega)$ and $\boldL_t^2(\partial\Omega)$ respectively.
\end{lemma}
\begin{proof}
We proceed as in~\cite{ArnoldFalkWinther:10}. Let $u\in\boldL^2(\Omega)$ and $u_n\in X_n$. Since $\pi^X_n$ is a projector
it follows
\begin{align*}
\|(1-\pi^X_n)u\|_{\boldL^2(\Omega)} &= \|(1-\pi^X_n)(u-u_n)\|_{\boldL^2(\Omega)}\\
&\leq \big(1+ \sup_{n\in\setN} \|\pi^X_n\|_{L(\boldL^2(\Omega))}\big) \|u-u_n\|_{\boldL^2(\Omega)}
\end{align*}
and hence
\begin{align*}
\|(1-\pi^X_n)u\|_{\boldL^2(\Omega)} &\leq \big(1+ \sup_{n\in\setN} \|\pi^X_n\|_{L(\boldL^2(\Omega))} \big)
\inf_{u_n\in X_n} \|u-u_n\|_{\boldL^2(\Omega)}.
\end{align*}
Since $X$ is densely embedded in $\boldL^2(\Omega)$ and due to Assumption~\ref{ass:XnI} the claim follows for $\pi^X_n$.
The claims for $\pi^Y_n$ and $\pi^Z_n$ follow like-wise.
\end{proof}

\begin{lemma}\label{lem:discnormP}
Let Assumptions~\ref{ass:eps}, \ref{ass:Domain}, \ref{ass:XnI} and \ref{ass:XnII} hold true. Then
\begin{subequations}
\begin{align}
\lim_{n\in\setN}\inf_{u\in X_n\setminus\{0\}}\|(1-\pi_n)P_Vu\|_X/\|u\|_X&=0,\\
\lim_{n\in\setN}\inf_{u\in X_n\setminus\{0\}}\|(1-\pi_n)P_{W_1}u\|_X/\|u\|_X&=0,\\
\lim_{n\in\setN}\inf_{u\in X_n\setminus\{0\}}\|(1-\pi_n)P_{W_2}u\|_X/\|u\|_X&=0.
\end{align}
\end{subequations}
\end{lemma}
\begin{proof}
We proceed as in~\cite{ArnoldFalkWinther:10}. 
Let $u_n\in X_n$. Due to $\curl P_{W_2}u_n=0$, $\tr_{\nv\times}P_{W_2}u_n=0$ and Assumption~\ref{ass:XnII} it
hold
\begin{align*}
\curl \pi_n P_{W_2}u_n=\pi^Y_n \curl P_{W_2}u_n=0
\end{align*}
and
\begin{align*}
\tr_{\nv\times} \pi_n P_{W_2}u_n=\pi^Z_n \tr_{\nv\times} P_{W_2}u_n=0.
\end{align*}
Hence
\begin{align*}
\|(1-\pi_n)P_{W_2}u_n\|_X &= \|(1-\pi_n)P_{W_2}u_n\|_{\boldL^2(\Omega)} = \|(1-\pi_n)(1-P_{W_2})u_n\|_{\boldL^2(\Omega)}\\
& \leq \|(1-\pi^X_n)E(1-P_{W_2})\|_{L(X,\boldL^2(\Omega)} \|u_n\|_X.
\end{align*}
Since $E|_{\ran (1-P_{W_2})}=E|_{H(\curl,\div^0,\tr_{\nv\times};\Omega)}$ is compact due to
\eqref{eq:CostabelDomain} and $1-\pi^X_n$ tends point-wise to zero it follows
$\lim_{n\in\setN}\|(1-\pi^X_n)E(1-P_{W_2})\|_{L(X,\boldL^2(\Omega))}=0$.\\

We compute
\begin{align*}
\curl \pi_n P_V u_n &= \pi^Y_n \curl P_V u_n = \pi^Y_n \curl (P_V+P_{W_1}+P_{W_1}) u_n\\
&= \pi^Y_n \curl u_n = \curl u_n = \curl (P_V+P_{W_1}+P_{W_1}) u_n = \curl P_V u_n
\end{align*}
and hence
\begin{align*}
\|(1-\pi_n)P_Vu_n\|_X^2 &= \|(1-\pi_n)P_Vu_n\|_{\boldL^2(\Omega)}^2
+ \|\tr_{\nv\times}(1-\pi_n)P_Vu_n\|_{\boldL_t^2(\partial\Omega)}^2.
\end{align*}
We estimate the first term
\begin{align*}
\|(1-\pi_n)P_Vu_n\|_{\boldL^2(\Omega)} &\leq \|(1-\pi^X_n)EP_V\|_{L(X,\boldL^2(\Omega))} \|u_n\|_X.
\end{align*}
As previously we obtain $\lim_{n\in\setN}\|(1-\pi^X_n)EP_V\|_{L(X,\boldL^2(\Omega))}=0$.
We estimate the second term
\begin{align*}
\|\tr_{\nv\times}(1-\pi_n)P_Vu_n\|_{\boldL_t^2(\partial\Omega)}
&= \|(1-\pi^Z_n)\tr_{\nv\times}P_Vu_n\|_{\boldL_t^2(\partial\Omega)} \\
&\leq \|(1-\pi^Z_n)\tr_{\nv\times}P_V\|_{L(X,\boldL_t^2(\partial\Omega))} \|u_n\|_X.
\end{align*}
Due to \eqref{eq:Vtraceregularity} $\tr_{\nv\times}|_V$ is compact, $(1-\pi^Z_n)$ tends point-wise to zero and hence
\begin{align*}
\lim_{n\in\setN} \|(1-\pi^Z_n)\tr_{\nv\times}P_V\|_{L(X,\boldL_t^2(\partial\Omega))}=0.
\end{align*}
The claim for $P_{W_1}$ follows from $P_{W_1}=1-P_V-P_{W_2}$.
\end{proof}

\begin{corollary}\label{cor:Tncompatible}
Let Assumptions~\ref{ass:eps}, \ref{ass:Domain}, \ref{ass:XnI} and \ref{ass:XnII} hold true.
Let $T_n(\lambda)\in L(X_n)$ be defined as $T_n(\lambda):=\pi_n T(\lambda)|_{X_n}$ for each $\lambda\in\setC\setminus\{0\}$.
Then
\begin{align}
\lim_{n\in\setN} \|T(\lambda)-T_n(\lambda)\|_n = 0
\end{align}
for each $\lambda\in\setC\setminus\{0\}$.
\end{corollary}
\begin{proof}
Follows from the definition of $T(\lambda)$, the triangle inequality and Lemma~\ref{lem:discnormP}.
\end{proof}

\begin{theorem}\label{thm:SpecAppr}
Let Assumptions~\ref{ass:eps}, \ref{ass:mu}, \ref{ass:Domain} and \ref{ass:UCP} hold true.
Let $X$, $A_X(\cdot)$ and $T(\cdot)$ be as defined in~\eqref{eq:X}, \eqref{eq:DefAX}
and~\eqref{eq:T} respectively. Let Assumptions~\ref{ass:XnI} and~\ref{ass:XnII} hold true.
Then $A_X(\cdot)\colon\setC\setminus\{0\} \to L(X)$ is a holomorphic weakly $T(\cdot)$-coercive operator function
with non-empty resolvent set and the sequence of Galerkin approximations
$\big(P_nA_X(\cdot)|_{X_n}\colon \setC\setminus\{0\} \to L(X_n)\big)_{n\in\setN}$ is $T(\cdot)$-compatible.
Thus~\cite[Corollary~2.8]{Halla:19Tcomp} is applicable.
\end{theorem}
\begin{proof}
Follows from Theorem~\ref{thm:AwTc}, Corollary~\ref{cor:ressetnonempty} and Corollary~\ref{cor:Tncompatible}.
\end{proof}

\section{Weak $T(\cdot)$-coercivity of the modified Stekloff operator function}\label{sec:modified}
First we introduce the modified electromagnetic Stekloff eigenvalue problem proposed in~\cite{CamanoLacknerMonk:17}
as holomorphic operator function eigenvalue problem. We proceed as in Section~\ref{sec:AFredholm}.
In Theorem~\ref{thm:VWtilde} we report an apt decomposition of the respective Hilbert space into two subspaces.
Next we introduce in~\eqref{eq:Ttilde} an operator $\tilde T$ as an apt sign change on the subspaces.
In Theorem~\ref{thm:AwTctilde} we report the weak $\tilde T$-coercivity of the modified Stekloff operator function.
In Corollary~\ref{cor:ressetnonemptytilde} we deduce convenient properties of the spectrum in $\setC$.
We report in a companion article \cite{Halla:19StekloffExist} the existence and behavior of eigenvalues for purely real,
symmetric $\mu$ and $\epsilon$, i.e.\ in the selfadjoint case.
In Subsection~\ref{subsec:Aux} we introduce a formulation with an auxiliary variable, which implicitly realizes the
action of the operator $S$ and prove respective properties.\\

The modified electromagnetic Stekloff eigenvalue problem is to
\begin{align}\label{eq:EVPtilde}
\text{find}\quad (\lambda,u)\in \setC\times H(\curl;\Omega)
\setminus\{0\}\quad\text{so that}\quad \tilde A(\lambda)u=0,
\end{align}
whereby $\tilde A(\lambda)\in L\big(H(\curl;\Omega)\big)$ is defined through
\begin{align}
\begin{aligned}
\spl \tilde A(\lambda)&u,u'\spr_{H(\curl;\Omega)}:=
\spl \mu^{-1} \curl u,\curl u'\spr_{\boldL^2(\Omega)}
-\omega^2\spl\epsilon u,u'\spr_{\boldL^2(\Omega)}\\
&-\lambda \spl Su,Su'\spr_{\boldL^2_t(\partial\Omega)}
\quad\text{for all }u,u'\in H(\curl;\Omega), \lambda\in\setC
\end{aligned}
\end{align}
and $S$ is as defined in~\eqref{eq:DefS}.
We note again that the sign of $\lambda$ herein is reversed compared to~\cite{CamanoLacknerMonk:17}.
Also, we employ $\tr_{\nv\times}u$ opposed to $u_\nv =\tr_{\nv\times}u\times\nv$ in~\cite{CamanoLacknerMonk:17}
and hence we employ through $S$ a map onto gradient functions opposed to a map onto curl functions
as in~\cite{CamanoLacknerMonk:17}.
As in Section~\ref{sec:AFredholm} we introduce apt subspaces of $H(\curl;\Omega)$:
\begin{subequations}
\begin{align}
\tilde V&:=H(\curl,\div^0,\tr_{\nv\cdot}^0;\Omega),\\
\tilde W&:=H(\curl^0;\Omega)=\nabla H^1(\Omega).
\end{align}
\end{subequations}

\begin{theorem}\label{thm:VWtilde}
It holds
\begin{align}
H(\curl;\Omega)=\tilde V\oplus^{\bot_{H(\curl;\Omega)}} \tilde W,
\end{align}
i.e.\ the orthogonal projection operators $P_{\tilde V}, P_{\tilde W} \in L\big(H(\curl;\Omega)\big)$ satisfy
$\ran P_{\tilde V}=\tilde V, \ran P_{\tilde W}=\tilde W$,
$\tilde W=\ker P_{\tilde V}$, $\tilde V=\ker P_{\tilde W}$,
$u=P_{\tilde V}u+P_{\tilde W}u$ for each $u\in H(\curl;\Omega)$
and
\begin{align}\label{eq:scptildeXtilde}
\begin{aligned}
\spl u,u'\spr_{\tilde X}&:=\spl P_{\tilde V}u,P_{\tilde V}u'\spr_{H(\curl;\Omega)}
+\spl P_{\tilde W}u,P_{\tilde W}u'\spr_{H(\curl;\Omega)}
=\spl u,u'\spr_{H(\curl;\Omega)}
\end{aligned}
\end{align}
for all $u,u'\in H(\curl;\Omega)$.
\end{theorem}
\begin{proof}
All properties are due to the orthogonal decomposition.
\end{proof}

We observe $\tilde W\subset \ker S$. We proceed further as in Section~\ref{sec:AFredholm}. Let
\begin{align}\label{eq:Xtilde}
\tilde X:=H(\curl;\Omega), \qquad
\spl\cdot,\cdot\spr_{\tilde X}\quad\text{as defined in~\eqref{eq:scptildeXtilde}}.
\end{align}
Let further $\tilde A_{\tilde X}(\cdot), \tilde A_c, \tilde A_\epsilon, \tilde A_{l^2}, \tilde A_{\tr}
\in L(\tilde X)$ be defined through
\begin{subequations}
\begin{align}
\label{eq:DefAXtilde}
\spl \tilde A_{\tilde X}(\lambda) u,u'\spr_{\tilde X}&:=\spl \tilde A(\lambda)u,u'\spr_{H(\curl;\Omega)}
\quad\text{for all }u,u'\in \tilde X,\lambda\in\setC,\\
\spl \tilde A_c u,u'\spr_{\tilde X}&:=\spl \mu^{-1}\curl u,\curl u'\spr_{\boldL^2(\Omega)}
\quad\text{for all }u,u'\in \tilde X,\\
\spl \tilde A_\epsilon u,u'\spr_{\tilde X}&:=\spl \epsilon u,u'\spr_{\boldL^2(\Omega)}
\quad\text{for all }u,u'\in \tilde X,\\
\spl \tilde A_{l^2} u,u'\spr_{\tilde X}&:=\spl u,u'\spr_{\boldL^2(\Omega)}
\quad\text{for all }u,u'\in \tilde X,\\
\spl \tilde A_{\tr} u,u'\spr_{\tilde X}&:=\spl Su,Su'\spr_{\boldL^2_t(\partial\Omega)}
\quad\text{for all }u,u'\in \tilde X.
\end{align}
\end{subequations}
From the definitions of $\tilde V$, $\tilde W$ and $\tilde W\subset \ker S$ we deduce that
\begin{align}\label{eq:AVWWtilde}
\begin{aligned}
\tilde A_{\tilde X}(\lambda)&=
(P_{\tilde V}+P_{\tilde W})(\tilde A_c-\omega^2 \tilde A_\epsilon-\lambda \tilde A_{\tr})(P_{\tilde V}+P_{\tilde W})\\
&={\color{blue}P_{\tilde V} \tilde A_c P_{\tilde V}}
-\omega^2 P_{\tilde V}\tilde A_\epsilon P_{\tilde V}
-\lambda P_{\tilde V}\tilde A_{\tr} P_{\tilde V}
{\color{blue}-\omega^2 P_{\tilde W}\tilde A_\epsilon P_{\tilde W}}\\
&-\omega^2 (P_{\tilde W}\tilde A_\epsilon P_{\tilde V}+P_{\tilde V}\tilde A_\epsilon P_{\tilde W}).
\end{aligned}
\end{align}
If we identify $\tilde X\sim \tilde V\times \tilde W$ and $\tilde X \ni u\sim (v,w)\in \tilde V\times \tilde W$,
we can identify $\tilde A_{\tilde X}(\lambda)$ with the block operator
\begin{align}\label{eq:Ablocktilde}
\bpm
{\color{blue}P_{\tilde V} \tilde A_c|_{\tilde V}} - P_{\tilde V}(\omega^2 \tilde A_\epsilon +\lambda \tilde A_{\tr}) |_{\tilde V}
&-\omega^2 P_{\tilde V}\tilde A_\epsilon |_{\tilde W}\\
-\omega^2 P_{\tilde W}\tilde A_\epsilon |_{\tilde V}
&{\color{blue}-\omega^2 P_{\tilde W} \tilde A_\epsilon|_{\tilde W}}&
\epm.
\end{align}
We color highlighted in \eqref{eq:AVWWtilde} and \eqref{eq:Ablocktilde} the operators which are not compact. This leads
us to define a test function operator in the following way. Let
\begin{align}\label{eq:Ttilde}
\tilde T&:=P_{\tilde V}-\omega^{-2}P_{\tilde W}.
\end{align}
Obviously $\tilde T\in L(\tilde X)$ is bijective with $\tilde T^{-1}=P_{\tilde V}-\omega^2 P_{\tilde W}$.

\begin{theorem}\label{thm:AwTctilde}
Let $\epsilon$ suffice Assumption~\ref{ass:eps}, $\mu$ suffice Assumption~\ref{ass:mu}
and $\Omega$ suffice Assumption~\ref{ass:Domain}.
Thence $\tilde A_{\tilde X}(\cdot)\colon \setC\to L(\tilde X)$ is weakly $\tilde T$-coercive.
\end{theorem}
\begin{proof}
Let $\lambda\in\setC$. Set
\begin{align*}
A_1&:=P_{\tilde V} \tilde A_c P_{\tilde V}
+P_{\tilde V} \tilde A_{l^2} P_{\tilde V}
-\omega^2 P_{\tilde W} \tilde A_\epsilon P_{\tilde W}
\end{align*}
and
\begin{align*}
A_2&:=-P_{\tilde V} \tilde A_{l^2} P_{\tilde V}
-\omega^2 P_{\tilde V} \tilde A_\epsilon P_{\tilde V}
-\lambda P_{\tilde V}\tilde A_{\tr} P_{\tilde V}
-\omega^2 (P_{\tilde W}\tilde A_\epsilon P_{\tilde V}+P_{\tilde V}\tilde A_\epsilon P_{\tilde W}).
\end{align*}
so that $\tilde A_{\tilde X}(\lambda)=A_1+A_2$. Operator $A_2$ is compact due to \eqref{eq:CostabelDomain}
and \eqref{eq:Vtraceregularity} and hence so is $T(\lambda)^*A_2$. It is straight forward to see
\begin{align*}
\Re(\spl A_1u,\tilde Tu \spr_{\tilde X}) \geq \min(1,c_\epsilon,c_\mu) \|u\|_{\tilde X}^2,
\end{align*}
i.e.\ $\tilde T^*A_1$ is coercive.
\end{proof}

As in~\cite{CamanoLacknerMonk:17} we impose an additional assumption.
\begin{assumption}\label{ass:kappa}
Let $\tilde A_{\tilde X}(0)$ be injective.
\end{assumption}

\begin{corollary}\label{cor:ressetnonemptytilde}
Let Assumptions~\ref{ass:eps}, \ref{ass:mu}, \ref{ass:Domain} and \ref{ass:kappa} hold true.
Then $\tilde A_{\tilde X}(\lambda)$ is bijective for all $\lambda\in\setC$ with $\Im(\lambda)<0$ and $\lambda=0$.
The spectrum of $\tilde A_{\tilde X}(\cdot)$ in $\setC$ consists of an at most countable set of eigenvalues
with finite algebraic multiplicity which have no accumulation point in $\setC$.
\end{corollary}
\begin{proof}
Let $\lambda\in\setC$ with $\Im(\lambda)<0$ and $u \in X$ be so that $\tilde A_{\tilde X}(\lambda)u=0$. It follows
\begin{align*}
0=-\Im(\spl \tilde A_{\tilde X}(\lambda)u,u)\spr_{\tilde X})
\geq-\Im(\lambda)\|Su\|_{\boldL^2_t(\partial\Omega)}^2
\end{align*}
and hence $\tilde A_{\tilde X}(0)u=\tilde A_{\tilde X}(\lambda)u=0$. Due to Assumption~\ref{ass:kappa} it follows
$u=0$, i.e.\ $\tilde A_{\tilde X}(\lambda)$ is injective.
From Theorem~\ref{thm:AwTctilde} we know that $\tilde A_{\tilde X}(\lambda)$ is Fredholm with index zero for all
$\lambda\in\setC$ and hence $\tilde A_{\tilde X}(\lambda)$ is bijective, if $\Im(\lambda)<0$ or $\lambda=0$.
For the remaining claim see the proof of Corollary~\ref{cor:ressetnonemptytilde}.
\end{proof}

\subsection{Auxiliary formulation}\label{subsec:Aux}
A Galerkin approximation to~\eqref{eq:EVPtilde} doesn't yield a computational method yet, because the term
$\spl Su_n,Su_n'\spr_{\boldL^ 2_t(\partial\Omega)}$ needs to evaluated. Therefore we proceed as
in~\cite{CamanoLacknerMonk:17} and introduce an auxiliary variable. To this end let
\begin{align}\label{eq:Xtt}
\Zt&:=H^1_*(\partial\Omega), \quad
\spl \cdot,\cdot \spr_{\Zt}:=\spl\nabla_\partial\cdot,\nabla_\partial\cdot\spr_{\boldL^2_t(\partial\Omega)},\\
\Xtt&:=\Xt\times\Zt, \quad \spl (u,z),(u',z') \spr_{\Xtt}:=\spl u,u' \spr_{\Xt}+\spl z,z' \spr_{\Zt}
\end{align}
for all $(u,z),(u',z')\in\Xtt$ and for $l\in\{0,1\}$ let
\begin{align}\label{eq:Att}
\begin{aligned}
\spl \Att^l(\lambda)(u,z),(u',z')\spr_{\Xtt}&:=
\spl \mu^{-1} \curl u,\curl u'\spr_{\boldL^2(\Omega)}
-\omega^2\spl\epsilon u,u'\spr_{\boldL^2(\Omega)}\\
&+\lambda \spl z,\div_\partial\tr_{\nv\times} u'\spr_{H^{1}(\partial\Omega)\times H^{-1}(\partial\Omega)}\\
&+\lambda^l \spl \div_\partial\tr_{\nv\times} u,z'\spr_{H^{-1}(\partial\Omega)\times H^{1}(\partial\Omega)}\\
&+\lambda^l \spl \nabla_\partial z,\nabla_\partial z'\spr_{\boldL^2_t(\partial\Omega)}
\end{aligned}
\end{align}
for all $(u,z),(u',z')\in \Xtt$ and $\lambda\in\setC$. If the coefficients $\mu$, $\epsilon$ are real and symmetric,
the choice $l=1$ preserves the self adjointness of~\eqref{eq:Att}. This is of advantage, if one chooses to implement a
discretization which is based directly on~\eqref{eq:Att}. On the other hand if one aims to build the Schur-complement
with respect to the second component in a later discretization step, then the choice $l=0$ leads to no restriction on
$\lambda$. Let
\begin{align}\label{eq:Lambdal}
\Lambda_0:=\setC, \qquad \Lambda_1:=\setC\setminus\{0\}.
\end{align}

\begin{lemma}\label{lem:relAtAtt}
If $(\lambda,u)\in\setC\times\Xt\setminus\{0\}$ so that $\At(\lambda)u=0$, then $\Att^l(\lambda)(u,z)=0$ with
$z\in\Zt$ so that $Su=\nabla_\partial z$. Vice-versa, if $(\lambda,(u,z))\in\Lambda_l\times\Xtt\setminus\{0\}$
so that $\Att^l(\lambda)(u,z)=0$, then $Su=\nabla_\partial z$ and $\At(\lambda)u=0$.
\end{lemma}
\begin{proof}
Let $(\lambda,u)\in\setC\times\Xt\setminus\{0\}$ so that $\At(\lambda)u=0$ and $z\in\Zt$ be so that
$\nabla_\partial z=Su$. It follows
\begin{align*}
0&=\spl \mu^{-1} \curl u,\curl u'\spr_{\boldL^2(\Omega)} -\omega^2\spl\epsilon u,u'\spr_{\boldL^2(\Omega)}
-\lambda \spl Su,Su'\spr_{\boldL^2_t(\partial\Omega)}\\
&=\spl \mu^{-1} \curl u,\curl u'\spr_{\boldL^2(\Omega)} -\omega^2\spl\epsilon u,u'\spr_{\boldL^2(\Omega)}
-\lambda \spl \nabla_\partial z,Su'\spr_{\boldL^2_t(\partial\Omega)}\\
&=\spl \mu^{-1} \curl u,\curl u'\spr_{\boldL^2(\Omega)} -\omega^2\spl\epsilon u,u'\spr_{\boldL^2(\Omega)}
+\lambda \spl z,\div_\partial Su'\spr_{H^{1}(\partial\Omega)\times H^{-1}(\partial\Omega)}\\
&=\spl \mu^{-1} \curl u,\curl u'\spr_{\boldL^2(\Omega)} -\omega^2\spl\epsilon u,u'\spr_{\boldL^2(\Omega)}
+\lambda \spl z,\div_\partial\tr_{\nv\times} u'\spr_{H^{1}(\partial\Omega)\times H^{-1}(\partial\Omega)}
\end{align*}
for each $u'\in\Xt$. It follows further
\begin{align*}
0=&\spl \div_\partial\tr_{\nv\times} u,z'\spr_{H^{-1}(\partial\Omega)\times H^{1}(\partial\Omega)}
+\spl \nabla_\partial z,\nabla_\partial z'\spr_{\boldL^2_t(\partial\Omega)}
\end{align*}
for each $z'\in\Zt$ from the definition of $S$ and $z$. The reverse direction follows like-wise.
\end{proof}

Let $B\in L(\Zt,\Xt)$ so that
\begin{align}
\spl Bz,u \spr_{\Xt} := \spl z,\div_\partial\tr_{\nv\times} u\spr_{H^{1}(\partial\Omega)\times H^{-1}(\partial\Omega)}
\end{align}
for all $z\in\Zt$, $u\in\Xt$. Then $\Att^l(\lambda)$ admits the block representation
\begin{align}
\Att^l(\lambda) =
\bpm \At_c-\omega^2\At_\epsilon & \lambda B \\ \lambda^l B^* & \lambda^l \I_{\Zt} \epm.
\end{align}
This leads us to define
\begin{align}\label{eq:Ttt}
\Ttt^l(\lambda):= \bpm \Tt &\\& \ol{\lambda}^{-l}\I_{\Zt} \epm, \quad \lambda\in\Lambda_l.
\end{align}

\begin{theorem}\label{thm:AwTctt}
Let $\epsilon$ suffice Assumption~\ref{ass:eps}, $\mu$ suffice Assumption~\ref{ass:mu} and $\Omega$ suffice
Assumption~\ref{ass:Domain}. Thence $\Att^l(\cdot)\colon \Lambda_l\to L(\Xtt)$ is weakly $\Ttt^l(\cdot)$-coercive.
\end{theorem}
\begin{proof}
Let
\begin{align*}
A_1:=\bpm P_{\Vt}(\At_c+\At_{l^2})P_{\Vt} -\omega^2P_{\Wt} \At_\epsilon P_{\Wt} &\\& \lambda^l \I_{\Zt} \epm.
\end{align*}
and
\begin{align*}
A_2:=\bpm -P_{\Vt}(\omega^2\At_\epsilon+\At_{l^2})P_{\Vt} -\omega^2( P_{\tilde V} \At_\epsilon P_{\tilde W}+
P_{\tilde W} \At_\epsilon P_{\tilde V}) & \lambda B\\ \lambda^l B^*& \epm.
\end{align*}
so that $\Att^l(\lambda)=A_1+A_2$. It follows
\begin{align*}
\Re (\spl A_1(u,z),\Ttt^l(\lambda)(u,z) \spr_{\Xtt}) \geq \min(1,c_\mu,c_\epsilon) \|(u,z)\|_{\Xtt}^2
\end{align*}
for each $(u,z)\in\Xtt$, i.e.\ $\Ttt^l(\lambda)^*A_1$ is coercive. Let $\iota\in L\big(H^{-1/2}(\partial\Omega),
H^{1/2}(\partial\Omega)\big)$ be the isomorphism so that
$\spl \phi,\phi' \spr_{H^{1/2}(\partial\Omega)\times H^{-1/2}(\partial\Omega)}
=\spl \phi,\iota \phi' \spr_{H^{1/2}(\partial\Omega)}$ for all $\phi\in H^{1/2}(\partial\Omega)$ and
$\phi'\in H^{-1/2}(\partial\Omega)$. Let $E\in L\big(H^{1}(\partial\Omega),H^{1/2}(\partial\Omega)\big)$ be the embedding
operator. Then
\begin{align*}
\spl Bz,u \spr_{\Xt} &= \spl z,\div_\partial\tr_{\nv\times} u\spr_{H^{1}(\partial\Omega)\times H^{-1}(\partial\Omega)} \\
&= \spl Ez,\div_\partial\tr_{\nv\times} u\spr_{H^{1/2}(\partial\Omega)\times H^{-1/2}(\partial\Omega)} \\
&= \spl Ez,\iota\div_\partial\tr_{\nv\times} u\spr_{H^{1/2}(\partial\Omega)} \\
&= \spl (\iota\div_\partial\tr_{\nv\times})^*Ez, u\spr_{\Xt},
\end{align*}
i.e.\ $B=(\iota\div_\partial\tr_{\nv\times})^*E$. Since $E$ is compact, so are $B$ and $B^*$. The remaining terms of $A_2$
are compact due to \eqref{eq:CostabelDomain}. Hence $\Ttt^l(\lambda)^*A_2$ is compact too.
\end{proof}

\begin{corollary}\label{cor:ressetnonemptytt}
Let Assumptions~\ref{ass:eps}, \ref{ass:mu}, \ref{ass:Domain} and \ref{ass:kappa} hold true.
Then $\Att^l(\lambda)$ is bijective for all $\lambda\in\setC$ with $\Im(\lambda)<0$.
\end{corollary}
\begin{proof}
Follows from Theorem~\ref{thm:AwTctt}, Lemma~\ref{lem:relAtAtt} and Corollary~\ref{cor:ressetnonemptytilde}.
\end{proof}

\section{Compatible approximation of the modified Stekloff eigenvalue problem}\label{sec:approximationMod}
In this section we discuss Galerkin approximations of $\Att^l(\cdot)$. We proceed as in Section~\ref{sec:approximation}.
We embrace the basic Assumptions~\ref{ass:XnItilde}, \ref{ass:ZnItilde} and in Assumption~\ref{ass:XnIItilde} the
existence of uniformly bounded commuting projections like in~\cite{ArnoldFalkWinther:10}.
We report in Corollary~\ref{cor:Tncompatiblett} that for Galerkin approximations which satisfy these three assumptions,
we can construct a sequence of operator functions $\Ttt^l_n(\lambda)\in L(\Xtt_n)$ which converges to $\Ttt^l(\lambda)$
in discrete norm~\eqref{eq:discretenorm} at each $\lambda\in\Lambda_l$. Consequently we report in
Theorem~\ref{thm:SpecApprtt} that the abstract framework of~\cite{Halla:19Tcomp} is applicable. Finally, we discuss some
topics concerning the computational implementation.\\

Consider the following basic assumptions.
\begin{assumption}\label{ass:XnItilde}
Let $(\Xt_n)_{n\in\setN}$ be so that $\Xt_n\subset \Xt$, $\dim \Xt_n<\infty$ for each $n\in\setN$ and
\begin{align}
\lim_{n\in\setN}\inf_{u'\in \Xt_n}\|u-u'\|_{\Xt}=0\text{ for each }u\in \Xt.
\end{align}
\end{assumption}

\begin{assumption}\label{ass:ZnItilde}
Let $(\Zt_n)_{n\in\setN}$ be so that $\Zt_n\subset \Zt$, $\dim \Zt_n<\infty$ for each $n\in\setN$ and
\begin{align}
\lim_{n\in\setN}\inf_{z'\in \Zt_n}\|z-z'\|_{\Zt}=0\text{ for each }z\in \Zt.
\end{align}
\end{assumption}

Let
\begin{align}
\Xtt_n:=\Xt_n\times \Zt_n.
\end{align}
Consider the following additional assumption.
\begin{assumption}\label{ass:XnIItilde}
There exists $(\pi^{\Xt}_n)_{n\in\setN}$ so that
\begin{subequations}
\begin{align}
\pi^{\tilde X}_n\in L\big(\boldL^2(\Omega)\big) \text{ is a projector with }\tilde X_n = \ran \pi^{\tilde X}_n,\\
\sup_{n\in\setN} \|\pi^{\tilde X}_n\|_{L(\boldL^2(\Omega))} < +\infty.
\end{align}
\end{subequations}
Let $\tilde Y:=\boldL^2(\Omega)$. There exist sequences $(\tilde Y_n, \pi^{\tilde Y}_n)_{n\in\setN}$ so that
\begin{subequations}
\begin{align}
\tilde Y_n\subset Y, \qquad \lim_{n\in\setN}\inf_{u'\in \tilde Y_n}\|u-u'\|_{\tilde Y}=0,\\
\pi^{\tilde Y}_n\in L(\tilde Y) \text{ is a projector with }\tilde Y_n \subset \ran \pi^{\tilde Y}_n,\\
\sup_{n\in\setN} \|\pi^{\tilde Y}_n\|_{L(\tilde Y)} < +\infty.
\end{align}
\end{subequations}
Denote $\tilde E\in L\big(\tilde X,\boldL^2(\Omega)\big)$ the embedding operator and set
\begin{align}
\tilde \pi_n:=\pi^{\tilde X}_n\tilde E.
\end{align}
Further let
\begin{align}
\curl\circ\tilde \pi_n u=\pi^{\tilde Y}_n\circ\curl u
\end{align}
for each $u\in \tilde X$.
\end{assumption}

\begin{lemma}\label{lem:QuasiBestApprtilde}
Let Assumptions~\ref{ass:XnItilde} and \ref{ass:XnIItilde} hold true. Then the projections $\pi^{\tilde X}_n$ and
$\pi^{\tilde Y}_n$ converge point-wise to the identity in $\boldL^2(\Omega)$.
\end{lemma}
\begin{proof}
Proceed as for Lemma~\ref{lem:QuasiBestAppr}.
\end{proof}

\begin{lemma}\label{lem:discnormPtilde}
Let Assumptions~\ref{ass:eps}, \ref{ass:Domain}, \ref{ass:XnItilde} and \ref{ass:XnIItilde} hold true. Then
\begin{subequations}
\begin{align}
\lim_{n\in\setN}\inf_{u\in \tilde X_n\setminus\{0\}}\|(1-\tilde \pi_n)P_{\tilde V}u\|_{\tilde X}/\|u\|_{\tilde X}&=0,\\
\lim_{n\in\setN}\inf_{u\in \tilde X_n\setminus\{0\}}\|(1-\tilde \pi_n)P_{\tilde W}u\|_{\tilde X}/\|u\|_{\tilde X}&=0.
\end{align}
\end{subequations}
\end{lemma}
\begin{proof}
Proceed as for Lemma~\ref{lem:discnormP}.
\end{proof}

\begin{corollary}\label{cor:Tncompatiblett}
Let Assumptions~\ref{ass:eps}, \ref{ass:Domain}, \ref{ass:XnItilde}, \ref{ass:ZnItilde} and \ref{ass:XnIItilde} hold
true. Let $\tilde T_n\in L(\tilde X_n)$ be defined as $\tilde T_n:=\tilde \pi_n \tilde T|_{\tilde X_n}$ and
$\Ttt^l_n(\lambda)\in L(\Xtt_n)$ as
\begin{align}
\Ttt^l_n(\lambda) := \bpm \Tt_n &\\& \ol{\lambda}^{-l}\I_{\Zt_n} \epm.
\end{align}
for $\lambda\in\Lambda_l$. Then
\begin{align}
\lim_{n\in\setN} \|\Ttt^l(\lambda)-\Ttt_n^l(\lambda)\|_n = 0
\end{align}
at each $\lambda\in\Lambda_l$.
\end{corollary}
\begin{proof}
Proceed as for Corollary~\ref{cor:Tncompatible}.
\end{proof}

\begin{theorem}\label{thm:SpecApprtt}
Let Assumptions~\ref{ass:eps}, \ref{ass:mu}, \ref{ass:Domain} and \ref{ass:kappa} hold true.
Let $\Xtt$, $\Att^l(\cdot)$, $\Ttt^l(\cdot)$ and $\Lambda_l$ be as defined in~\eqref{eq:Xtt}, \eqref{eq:Att},
\eqref{eq:Ttt} and \eqref{eq:Lambdal} respectively. Let Assumptions~\ref{ass:XnItilde}, \ref{ass:ZnItilde}
and~\ref{ass:XnIItilde} hold true.
Then $\Att^l(\cdot)\colon\Lambda_l\to L(\Xtt)$ is a holomorphic weakly $\Ttt^l(\cdot)$-coercive operator function
with non-empty resolvent set and the sequence of Galerkin approximations
$\big(\Ptt_n \Att^l(\cdot)|_{\Xtt_n}\colon \Lambda_l \to L(\Xtt_n)\big)_{n\in\setN}$ is $\Ttt^l(\cdot)$-compatible.
Thus Corollary~2.8 of~\cite{Halla:19Tcomp} is applicable.
\end{theorem}
\begin{proof}
Follows from Theorem~\ref{thm:AwTctt}, Corollary~\ref{cor:ressetnonemptytt} and Corollary~\ref{cor:Tncompatiblett}.
\end{proof}

Theorem~\ref{thm:SpecApprtt} tells that suitable Galerkin approximations to $\Att^l(\cdot)$ yield reliable approximations.
In particular, if $\Xt_n$ and $\Zt_n$ are chosen as finite element spaces with fixed polynomial degrees $p_{\Xt}$, $p_{\Zt}$
and decreasing mesh width $h(n)$, Theorem~\ref{thm:SpecApprtt} tells that one should choose $p_{\Xt}=p_{\Zt}$ to obtain
asymptoticly optimal convergence rates.

We move on and discuss further issues related to the computational implementation. We note that if
$\Xt_n\subset H(\curl,\tr_{\nv\times};\Omega)$, then the duality pairs in~\eqref{eq:Att} can be evaluated as integrals:
\begin{align*}
\spl z_n,\div_\partial\tr_{\nv\times} u_n\spr_{H^{1}(\partial\Omega)\times H^{-1}(\partial\Omega)}
= -\spl \nabla_\partial z_n,\tr_{\nv\times} u_n\spr_{\boldL^2_t(\partial\Omega)}.
\end{align*}
Let further for $u_n\in\Xt_n$, $z_n$ be the solution to find $z_n\in\Zt_n$ so that
\begin{align}
\spl \nabla_\partial z_n, \nabla_\partial z_n' \spr_{\boldL_t^2(\partial\Omega)}
= -\spl \div_\partial \tr_{\nv\times} u_n, z_n' \spr_{H^{-1}(\partial\Omega)\times H^1(\partial\Omega)}
\end{align}
for all $z_n'\in\Zt_n$ and set
\begin{align}\label{eq:DefSn}
S_nu:=\nabla_\partial z_n.
\end{align}
From the construction of $S_n$ it follows $S_n\in L\big(\Xt_n,\boldL_t^2(\partial\Omega)\big)$ and further
\begin{align}
S_nu=P_{\nabla_\partial}^n\tr_{\nv\times} u_n
\end{align}
for $u_n\in H(\curl,\tr_{\nv\times};\Omega)$ with $P_{\nabla_\partial}^n$ being the $\boldL_t^2(\partial\Omega)$-orthogonal
projection onto $\nabla\Zt_n$. Let further $\At_n(\lambda)\in L(\Xt_n)$ be defined by
\begin{align}\label{eq:Atn}
\begin{aligned}
\spl \At_n(\lambda)u_n,u'_n\spr_{\Xt}&:=
\spl \mu^{-1} \curl u_n,\curl u_n'\spr_{\boldL_t^2(\partial\Omega)}
-\omega^2\spl\epsilon u_n,u'_n\spr_{\boldL_t^2(\partial\Omega)}\\
&-\lambda \spl S_n u_n,S_n u'_n\spr_{\boldL_t^2(\partial\Omega)}
\quad\text{for all }u_n,u'_n\in \Xt_n, \lambda\in\setC,
\end{aligned}
\end{align}
i.e.\ $\At_n(\lambda)$ is the Schur-complement of $\Ptt_n \Att^0(\lambda)|_{\Xtt_n}$ with respect to $z_n\in\Zt_n$.
Obviously $\At_n(\cdot)$ is a Galerkin approximation with variational crime $S_n^*S_n\neq \Pt_nS^*S|_{\Xt_n}$
of $\At(\cdot)$. The approximation properties of $\At_n(\cdot)$ to $\At(\cdot)$ are already provided by our previous
analysis, i.e.\ our analysis technique avoided to the discuss the variational crime directly.
If further $\Xt_n\subset H(\curl,\tr_{\nv\times};\Omega)$, then
\begin{align}
\begin{aligned}
\spl S_n u_n,S_n u'_n\spr_{\boldL_t^2(\partial\Omega)}
&= \spl P_{\nabla_\partial}^n\tr_{\nv\times} u_n,P_{\nabla_\partial}^n\tr_{\nv\times} u'_n\spr_{\boldL_t^2(\partial\Omega)}\\
&= \spl P_{\nabla_\partial}^n\tr_{\nv\times} u_n,\tr_{\nv\times} u'_n\spr_{\boldL_t^2(\partial\Omega)}\\
&= \spl S_n u_n,\tr_{\nv\times} u'_n\spr_{\boldL_t^2(\partial\Omega)}.
\end{aligned}
\end{align}
Let $(z_n)_{n=1}^N$ be a basis of $\Zt_n$ and consider the matrix $M\in\setC^ {N\times N}$ with entries
\begin{align}
M_{n,m}:=\spl \nabla_\partial z_n, \nabla_\partial z_m \spr_{\boldL_t^2(\partial\Omega)}.
\end{align}
To implement the operator $S_n$, the matrix $M$ needs to be inverted. However, due to $\Zt_n\subset\Zt=
H^1_*(\partial\Omega)$ the matrix $M$ is dense. To obtain a sparse matrix $M$ the following procedure was suggested
in~\cite{CamanoLacknerMonk:17}. Let $\gamma>0$ be small and $\setK:=\spn \{1\}$ be the space of constant functions.
For $u_n\in\Xt_n$ let $z_n$ be the solution to find $z_n\in\Zt_n\oplus\setK\subset H^1(\partial\Omega)$ so that
\begin{align}\label{eq:EqSgamma}
\spl \nabla_\partial z_n, \nabla_\partial z_n' \spr_{\boldL_t^2(\partial\Omega)}
+\gamma\spl z_n,z_n'\spr_{L^2(\partial\Omega)}
= -\spl \div_\partial \tr_{\nv\times} u_n, z_n' \spr_{H^{-1}(\partial\Omega)\times H^1(\partial\Omega)}
\end{align}
for all $z_n'\in\Zt_n\oplus\setK\subset H^1(\partial\Omega)$ and set
\begin{align}\label{eq:DefSngamma}
S_n^ \gamma u_n:=\nabla_\partial z_n.
\end{align}
We analyze this modification in two steps. First we consider the perturbation of the sesquilinear form
$\spl \nabla_\partial \cdot, \nabla_\partial\cdot \spr_{\boldL_t^2(\partial\Omega)}$ to
$\spl \nabla_\partial \cdot, \nabla_\partial\cdot \spr_{\boldL_t^2(\partial\Omega)}
+\gamma\spl \cdot,\cdot\spr_{L^2(\partial\Omega)}$ on the space $\Zt_n\subset H^1_*(\partial\Omega)$. The analysis of
such a perturbation is straight forward and of magnitude $\gamma$. Secondly we note that the solution
$z_n\in\Zt_n\oplus\setK\subset H^1(\partial\Omega)$ to~\eqref{eq:EqSgamma} satisfies $\spl z_n,1\spr_{L^2(\partial\Omega)}
=0$, i.e.\ $z_n\in\Zt_n$. Thus a replacement of $\Zt_n\subset H^1_*(\partial\Omega)$ by
$\Zt_n\oplus\setK\subset H^1(\partial\Omega)$ doesn't change the respective solution to~\eqref{eq:EqSgamma} and hence
no additional error is produced.

\bibliographystyle{amsplain}
\bibliography{../../../bibliography}
\end{document}